\documentclass[11pt,reqno]{amsart}
\usepackage{amsmath}
\usepackage{amsfonts}

\usepackage{amscd}
\usepackage{latexsym}
\usepackage{amsthm}
\usepackage{mathrsfs}
\usepackage{amssymb} \usepackage{latexsym}
\usepackage{eufrak}
\usepackage{euscript}
\usepackage{epsfig}
\usepackage{graphics}
\usepackage{array}
\usepackage{enumerate}
\usepackage{dsfont}
\usepackage{color}
\usepackage{wasysym}
\usepackage{hyperref}
\usepackage{pdfsync}
\usepackage{ifpdf}
\usepackage{float}

\newcommand{\bel}[1]{\begin{equation}\label{#1}}

\newcommand{\be}{\begin{equation}}

\newcommand{\ba}{\begin{eqnarray}}
\newcommand{\ea}{\end{eqnarray}}

\newcommand{\qe}{\end{equation}}
\newcommand{\R}{{\mathbb R}}

\newcommand{\y}{(\Omega)}

\newcommand{\OM}{{\overline{\Omega}}}

\newcommand{\Hmm}[1]{\leavevmode{\marginpar{\tiny%
$\hbox to 0mm{\hspace*{-0.5mm}$\leftarrow$\hss}%
\vcenter{\vrule depth 0.1mm height 0.1mm width \the\marginparwidth}%
\hbox to
0mm{\hss$\rightarrow$\hspace*{-0.5mm}}$\\\relax\raggedright #1}}}

\theoremstyle{theorem}
\newtheorem{thm}{Theorem}[section]
\newtheorem{prop}{Proposition}[section]
\theoremstyle{example}
\newtheorem{example}{Example}[section]
\theoremstyle{corollary}
\newtheorem{coro}{Corollary}[section]
\theoremstyle{lemma}
\newtheorem{lemma}{Lemma}[section]
\theoremstyle{definition}
\newtheorem{defi}{Definition}[section]
\theoremstyle{proof}

\theoremstyle{remark}

\begin{document}

\title{Neumann Cheeger constants on graphs}

\author{Bobo Hua and Yan Huang}
\email{bobohua@fudan.edu.cn}
\address{School of Mathematical Sciences, LMNS, Fudan University, Shanghai 200433,
China; Shanghai Center for Mathematical Sciences, Fudan University, Shanghai 200433, China}
\email{yanhuang0509@gmail.com}
\address{School of Mathematical Sciences, Fudan University, Shanghai 200433, China.}

\begin{abstract} For any subgraph of a graph, the Laplacian with Neumann boundary condition was introduced by Chung and Yau \cite{ChungYau1994}. In this paper, motivated by the Riemannian case, we introduce the Cheeger constants for Neumann problems and prove corresponding Cheeger estimates for first nontrivial eigenvalues.
\end{abstract}
\maketitle

\section{Introduction}
\par
Let $(M,g)$ be a closed Riemannian manifold, i.e. compact and without boundary. The Cheeger constant of $M$, a version of isoperimetric constants, is defined as
$$h_M=\inf_\Omega\frac{\mathrm{Area}(\partial \Omega)}{\min\{\mathrm{Vol}(\Omega),\mathrm{Vol}(\Omega^c)\}},$$
where the infimum is taken over all precompact open submanifolds $\Omega\subset M$ with Lipschitz boundary $\partial\Omega.$¡¡ Here we denote by $\Omega^c:=M\setminus\Omega$ the complement of $\Omega,$ by $\mathrm{Vol}(\cdot)$ and $\mathrm{Area}(\cdot)$ the Riemannian volume and area respectively. In \cite{Cheeger1970}, Cheeger discovered a close relation between the geometric quantity, $h_M$, and the analytic quantity, the first nontrivial eigenvalue of the Laplace-Beltrami operator on $M.$ It is nowadays called the Cheeger estimate for the first nontrivial eigenvalue of the Laplacian.


For a compact Riemannian manifold $M$ with boundary $\partial M$, one usually needs to impose boundary conditions, either Dirichlet or Neumann, on the Laplace-Beltrami operator so as to obtain self-adjoint operators. In the same spirit of Cheeger \cite{Cheeger1970}, different geometric quantities concerning the boundary effects, called Dirichlet or Neumann Cheeger constants, can be defined to control the first (or nontrivial) eigenvalues of corresponding Laplacians, see e.g. \cite[Chapter 9]{GeometryAnalysis2012}. In this paper, we focus on Neumann boundary problems. The Neumann Cheeger constant of $(M,\partial M)$ is defined as
$$h_{N}(M)=\inf_{\Omega}\frac{\mathrm{Area}(\partial\Omega\cap \mathrm{int}(M))}{\min\{\mathrm{Vol}(\Omega),\mathrm{Vol}(\Omega^c)\}},$$ where the infimum is taken over all precompact open submanifolds $\Omega\subset M$ with Lipschitz boundary. From the definition one can see the part of $\partial \Omega$ contained in $\partial M$ does not play a role in the definition.
With the same method in \cite{Cheeger1970}, one can get the Neumann Cheeger estimate for $M$, namely
$$\lambda_{1,N}(M)\geq\frac{1}{4}h_{N}^2(M),$$
where $\lambda_{1,N}(M)$ is the first nontrivial eigenvalue of the Neumann Laplacian operator on $M$, see \cite[p.259]{Chavel1984}.

Inspired by Riemannian geometry, Cheeger type estimates have been generalized to the discrete setting. We recall some basic definitions of graphs. Let $V$ be a finite set which serves as the set of vertices of a graph and $\mu:V\times V\to [0,\infty), (x,y)\mapsto \mu_{xy}=\mu_{yx}$ be a symmetric weight function. This induces a graph structure, denoted by the pair $G:=(V,\mu)$, with the set of vertices $V$ and the set of edges $E$ which is defined as $\{x,y\}\in E$ if and only if $\mu_{xy}>0,$ in symbols $x\sim y.$ Note that we do allow self-loops in the graph, i.e. $x\sim x$ if $\mu_{xx}>0.$ In fact, $\mu$ induces a measure on $E.$ One can define a degree measure on $V$ accordingly
$$m_x:=\sum_{y\in V}\mu_{xy},\quad x\in V,$$
and denote by $m(\Omega):=\sum_{x\in\Omega}m_x$ the measure of $\Omega\subset V.$

Given $\Omega_1,\Omega_2\subset V,$ we denote by $E(\Omega_1,\Omega_2):=\{\{x,y\}\in E\mid x\in \Omega_1,y\in \Omega_2\}$ the set of edges between $\Omega_1$ and $\Omega_2$. Corresponding to Riemannian case, the Cheeger constant of a finite graph $G$ was introduced by Dodziuk \cite{Dodziuk1984} and  Alon \cite{AlonMilman1985} independently:
\begin{equation}
\label{df:Cheeger const}h_G=\inf_{\emptyset\neq\Omega\subsetneqq V}\frac{\mu(\partial\Omega)}{\min\{m(\Omega),m(\Omega^c)\}},\end{equation}
where $\Omega^c:=V\setminus\Omega$, $\partial \Omega:=E(\Omega,\Omega^c)$ denotes the edge boundary of $\Omega.$ Moreover, the Cheeger estimate reads as
\begin{equation}\label{df:Cheeger estimate}
\frac{1}{2}h_G^2\leq 1-\sqrt{1-h_G^2}\leq\lambda_1(G)\leq 2h_G,\end{equation}
where $\lambda_1(G)$ is the first nontrivial eigenvalue of the Laplacian operator on $V$, see e.g. \cite[Lemma 2.1, Theorem 2.3]{Chung97}.

Neumann Laplace operater on any subgraph of a graph was introduced by Chung and Yau \cite{ChungYau1994}, which has been studied by many authors, see e.g.  \cite{ChungGrahamYau96,ChungYau97,Tan1999}. Given a subset $\Omega\subset V,$ we denote by $\delta\Omega$ the vertex boundary of $\Omega$ consisting of vertices in $\Omega^c$ that are adjacent to some vertices in $\Omega.$ We define a measure $m'$ on $\delta\Omega$ as $m'_z:=\sum_{y\in \Omega, y\sim z}\mu_{zy},\forall z\in\delta\Omega.$ In notation, we write $\overline{\Omega}:=\Omega\cup \delta\Omega$ and $E_{\Omega}:=E(\Omega,\overline{\Omega}).$ The Neumann Laplace operator appears naturally from the following variational problem
\begin{equation}\label{df:energy}D_{\Omega}(f):=\sum_{e=\{x,y\}\in E_{\Omega}}(f(x)-f(y))^2\mu_{xy},\quad f:\overline{\Omega}\to\R,\end{equation} under the constraint that $$\sum_{x\in \Omega}f(x)^2m_x=1.$$ The critical point $f$ and the critical value $\lambda$ of the above problem satisfy the Neumann eigenvalue problem
\begin{equation}\label{neumannlaplace}
\left\{
\begin{aligned}
&\Delta f(x):=\frac{1}{m_x}\sum_{y\in V: y\sim x}\mu_{xy}(f(y)-f(x))=-\lambda f(x),\quad \forall\ x\in \Omega, \\
&\sum_{y\in\Omega: y\sim z}\mu_{yz}(f(y)-f(z))=0,\quad\forall\ z\in\delta \Omega. \\
\end{aligned}
\right.
\end{equation}
The second equation in \eqref{neumannlaplace} justifies the Neumann boundary condition of the Laplacian in the discrete setting.

From a different perspective, Neumann Laplacians emerge from Markov processes, i.e. the simple random walks with reflections on the boundaries, observed by Chung and Yau \cite{ChungYau97}: Given a subset $\Omega\subset V,$ consider the simple random walk starting at vertices in $\Omega.$ Once the walker reaches the boundary $\delta\Omega,$ it is reflected via the edges, chosen with probability according to edge weights, back into vertices in $\Omega$. By considering the boundary effects, one can show that this reflection process is equivalent to the simple random walk on a graph $\widetilde{\Omega}:=(\Omega,\widetilde{\mu})$ without boundary where $\widetilde{\mu}$ is the modified edge weights given by $$\widetilde{\mu}_{xy}=\mu_{xy}+\sum_{z\in\delta\Omega}\frac{\mu_{xz}\mu_{zy}}{m'_z},\quad \forall x,y\in \Omega.$$
Note that in the modified graph $\widetilde{\Omega}$ the boundary effects produce many self-loops and bridges ($\widetilde{\mu}_{xy}>0,$ for $x \nsim y$ near the boundary). For example, let $\Omega=\{v_1,v_2,v_3\}$ be a subset of the standard lattice graph $\mathbb{Z}^2$ with unit edge weights. Then by the reflection process, the modified graph $(\Omega,\widetilde{\mu})$ is depicted in Figure \ref{figure1}  together with its modified edge weights. For convenience, all the numbers depicted in the Figures below denote edge weights.
\begin{figure}[!h]
\includegraphics[height=6cm,width=12cm]{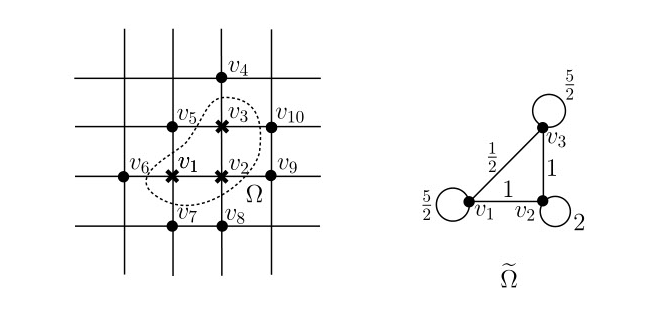}
\caption{}\label{figure1}
\end{figure}

By the above equivalence relation in \cite{ChungYau97}, the problems of Neumann Laplacians on subgraphs completely boil down to the problems of Laplacians on graphs without boundary conditions. They even share the same eigenvalues and eigenfunctions, see e.g. Lemma \ref{reducegraph} in this paper. Hence, by the Cheeger estimate \eqref{df:Cheeger estimate}, one has
\begin{equation}\label{eq:trivial estimate}
1-\sqrt{1-h_{\widetilde\Omega}^2}\leq\lambda_{1,N}(\Omega)\leq 2h_{\widetilde{\Omega}},\end{equation} where $\lambda_{1,N}(\Omega)$ is the first nontrivial eigenvalue of Neumann Laplacian on $\Omega$ and $h_{\widetilde{\Omega}}$ is the Cheeger constant of $\widetilde{\Omega}.$ As a Cheeger type estimate this result is already quite useful. However it involves a specific quantity $h_{\widetilde{\Omega}}$, derived from the modified graph $\tilde{\Omega}$ by random walk with reflection whose geometric meaning is not clear at the moment, has
no counterpart in the Riemannian setting. For our purpose, we would like
to obtain a Cheeger estimate via some geometric quantities obtained directly from the original data of the
graph, mimicking the Riemannian case.

Following the idea of Cheeger estimates for eigenvalues of Neumann Laplacians and treating the subset $\Omega\subset V$ as a manifold with boundary, we define the Neumann Cheeger constant of $\Omega$ as follows. Let $S\subset \overline{\Omega},$ we denote by $\partial_{\Omega}S:=\partial S\cap E(\Omega,\overline{\Omega})$ the relative boundary and $S^{\vee}:=\overline{\Omega}\setminus S$ the relative complement of $S$ in $\overline{\Omega}.$

\begin{defi}\label{neumanncheeger}
  The Neumann Cheeger constant of $\Omega$ in $(V,\mu)$ is defined as
  \begin{equation}\label{df:Neumann Cheeger}
    h_N(\Omega)=\inf_{S}\frac{\mu(\partial_\Omega S)}{\min\{m(S\cap \Omega),m(S^\vee\cap \Omega)\}},
  \end{equation} where the infimum is taken over all nonempty proper subsets $S$ of $\overline{\Omega}.$
\end{defi}
\noindent

We first show that $h_{N}(\Omega)$ is equal to a type of Sobolev constant, for the Riemannian case, see \cite[Theorem 9.6]{GeometryAnalysis2012}. Similar results for the discrete setting can be found in \cite[Theorem 2.6]{Chung97} and \cite[Lemma 5.14]{Chang2014}.

\begin{thm}\label{cheegerseqsobolev}
$$ h_N(\Omega)=\inf_{f}\frac{\sum_{e=\{x,y\}\in E_{\Omega}}|f(x)-f(y)|\mu_{xy}}{\inf_{c\in\R} \sum_{x\in \Omega}|f(x)-c|m_{x}},$$
where $f$ ranges over all real non-constant functions on $\overline{\Omega}$.
\end{thm}

Moreover, we prove that these two candidates of Cheeger constants, $h_{\widetilde{\Omega}}$ and $h_N(\Omega)$, are closely related to each other.
\begin{thm}\label{thm:main2}
  $h_{\widetilde\Omega}\leq h_{N}(\Omega)\leq 2h_{\widetilde\Omega}.$
\end{thm}
We give two examples to show the sharpness of both estimates. For the lower bound estimate, see Figure \ref{figure2}: Let $\Omega=\{v_1,v_2,v_3,v_4\}$, then we have $h_{\widetilde{\Omega}}=h_N(\Omega)=\frac{1}{201}$. For the upper bound estimate, we consider a path graph of length 2 as shown in Figure \ref{figure3}: Let $\Omega=\{v_1,v_3\},$ then $h_N(\Omega)=1$ and $h_{\widetilde{\Omega}}=\frac{1}{2}.$ Moreover, we characterise the conditions for the equalities in Theorem \ref{thm:main2} respectively, see
Proposition \ref{firstequality} and \ref{secondequality}.
\begin{figure}[!h]
\includegraphics[height=4cm,width=12cm]{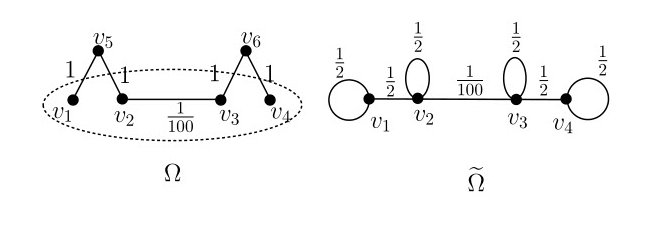}
\caption{}\label{figure2}
\end{figure}
\begin{figure}[!h]
\includegraphics[height=4cm,width=11cm]{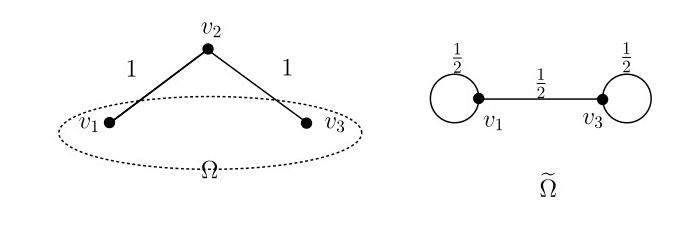}
\caption{}\label{figure3}
\end{figure}

Finally, we obtain the main result of the paper, the Neumann Cheeger estimate for the first nontrivial eigenvalue of Neumann Laplacian using $h_N(\Omega).$
\begin{thm}\label{NeumannCheegerestimates}
  \begin{equation}\label{eq:main estimate}2-\sqrt{4-h_N^2(\Omega)}\leq\lambda_{1,N}(\Omega)\leq2h_N(\Omega).\end{equation}
\end{thm}


For the lower bound estimates in \eqref{eq:trivial estimate} and \eqref{eq:main estimate}, concerning Theorem \ref{thm:main2}, it is hard to tell which one is better. For example in Figure \ref{figure3}, our lower bound estimate \eqref{eq:main estimate} is better than \eqref{eq:trivial estimate} by a scalar factor $2.$ As vague intuition, the reason is that the part of boundary effects, self-loops in $\widetilde\Omega,$ play no role in the numerator of the definition of $h_{\widetilde\Omega}$, i.e. one cannot cut the self-loops by that definition. However one can measure the boundary more effectively in $h_N(\Omega).$ From this point of view, the geometric idea rather gives some new insights on this problem. In summary, one gets
\begin{equation}\label{spectralgaplowerbound}
\max\left\{2-\sqrt{4-h_N^2(\Omega)},1-\sqrt{1-h_{\widetilde{\Omega}}^2}\right\}\leq\lambda_{1,N}(\Omega)\leq 2h_{\widetilde{\Omega}};\end{equation} in particular, this yields $$\max\left\{\frac{1}{4}h_N^2(\Omega),\frac12h_{\widetilde{\Omega}}^2\right\}\leq\lambda_{1,N}(\Omega).$$

As an application, the Neumann Cheeger estimate can be used to estimate the spectral gap of Dirichlet Laplace operator. From \cite[Corollary 3.1]{ChungOden2000} we know that the spectral gap of Dirichlet Laplace operator can be bounded from below by $\lambda_{1,N}(\Omega)$, namely
$$\lambda_{2,D}(\Omega)-\lambda_{1,D}(\Omega)\geq\lambda_{1,N}(\Omega),$$
where $\lambda_{1,D}(\Omega)$ and $\lambda_{2,D}(\Omega)$ are the first and second eigenvalue of Dirichlet Laplace operator on $\Omega$. Hence combining with the lower bound estimate in \eqref{spectralgaplowerbound}, we have a lower bound estimate for the spectral gap of Dirichlet Laplace operator on $\Omega$ using Cheeger constants.

The organization of the paper is as follows: In section 2, we recall some basic facts about Neumann Laplacians on graphs. In section 3, we characterise the properties of Neumann Cheeger constants. In section 4, we give the proof of the main result, i.e. Theorem \ref{NeumannCheegerestimates}.


\section{preliminaries}

Let $G=(V,\mu)$ be a connected graph and $\Omega$ a finite subgraph of $G$. We study the Laplacians with Neumann boundary conditions on $\Omega.$ To avoid the triviality, we assume that $\Omega$ has at least two vertices and non-empty vertex boundary, i.e. $\delta\Omega\neq\emptyset$. We order the eigenvalues of the Neumann Laplacian on $\Omega$ in the non-decreasing way:
$$0=\lambda_{0,N}(\Omega)\leq \lambda_{1,N}(\Omega)\leq\cdots\leq\lambda_{K-1,N}(\Omega),$$
where the subscripts $N$ indicate the Neumann eigenvalues and $K$ denotes the number of vertices in $\Omega$. We call $\lambda_{1,N}(\Omega)$ the first non-trivial
eigenvalue of the Neumann Laplacian on $\Omega$ (although it could be zero in cases). We denote by $\R^{\overline{\Omega}}$ the set of all real functions on $\overline{\Omega}.$ Due to variational principle, $\lambda_{1,N}(\Omega)$ can be characterized by the Reighley quotient, see \cite[pp.125]{Chung97},
\begin{eqnarray}\label{Reighley quotient}
  \lambda_{1,N}(\Omega)&=&\inf\left\{\frac{D_{\Omega}(f)}{\sum_{x\in\Omega}f^2(x)m_x}: 0\not\equiv f\in \R^{\overline{\Omega}},\sum_{x\in \Omega}f(x)m_x=0\right\}\\
  &=&\inf\left\{\frac{D_{\Omega}(f)}{\inf_{c\in\R}\sum_{x\in\Omega}|f(x)-c|^2m_x}:\ \mathrm{nonconstant}\ f\in \R^{\overline{\Omega}}\right\},\nonumber\end{eqnarray}
where the Dirichlet energy of $f,$ $D_{\Omega}(f),$ is defined as in \eqref{df:energy}. For simplicity, from now on we denote by $\sum_{e=\{x,y\}\in E_{\Omega}}$ the summation of edges in $E_{\Omega}$ where each edge is only counted once.

As pointed out by Chung and Yau \cite{ChungYau97}, the eigenvalue problem of Neumann Laplacian on $\Omega$ is equivalent to that on the modified graph $\widetilde{\Omega}=(\Omega,\widetilde{\mu})$ without boundary conditions. We provide a proof here for completeness. We denote by $\widetilde{m}_x:=\sum_{y\sim x,y\in \Omega}\widetilde{\mu}_{xy}$ the degree of $x$ in $\widetilde{\Omega}$ and one can see that it coincides with the degree in the original graph $G$ by direct calculation.
\begin{prop}\label{prop:degree}
  $$\widetilde{m}_x=m_x,\quad \forall x\in\Omega.$$
\end{prop}
\begin{proof}
  \begin{eqnarray*}
\widetilde{m}_x&=&\sum_{z\in\Omega}\mu_{zx}+\sum_{z\in\Omega}\sum_{y\in\delta\Omega}\frac{\mu_{xy}}{m'_y}\mu_{yz}\\
&=&\sum_{z\in\Omega}\mu_{zx}+\sum_{y\in\delta\Omega}\mu_{yx}\sum_{z\in\Omega}\frac{\mu_{yz}}{m'_y}\\
&=&m_x.\nonumber
\end{eqnarray*}
\end{proof}

\begin{lemma}\label{reducegraph}
If $(\lambda,u)$ is a pair of eigenvalue and eigenfunction of the Neumann Laplacian on $\Omega,$ i.e. satisfying \eqref{neumannlaplace} in the introduction, then $(\lambda,u)$ is an eigen-pair on $\widetilde{\Omega},$ i.e.
$$-\Delta u(x)=-\frac{1}{\widetilde{m}_x}\sum_{z\in\Omega}\widetilde\mu_{zx}(u(z)-u(x))=\lambda u(x),\quad x\in \Omega,$$ and visa versa.
\end{lemma}
\begin{proof}
For any $x\in\Omega$:
\begin{eqnarray*}
  \Delta u(x)&=&\frac{1}{m_x}\sum_{y\in\overline\Omega}\mu_{yz}(u(y)-u(x))\\
  &=&\frac{1}{m_x}\sum_{z\in\Omega}\mu_{zx}(u(z)-u(x))+\frac{1}{m_x}\sum_{y\in\delta\Omega}\mu_{yx}(u(y)-u(x)).\\
  \end{eqnarray*}
By the Neumann boundary condition, for $y\in\delta\Omega$, one checks that $u(y)=\frac{1}{m'_y}\Sigma_{z\in\Omega}\mu_{yz}u(z)$. Plugging this into the previous equation, we have
\begin{eqnarray*}
\Delta u(x)&=&\frac{1}{m_x}\sum_{z\in\Omega}\mu_{zx}(u(z)-u(x))+\frac{1}{m_x}\sum_{z\in\Omega}\sum_{y\in\delta\Omega}\frac{\mu_{yx}}{m'_y}\mu_{yz}(u(z)-u(x))\\
&=&\frac{1}{m_x}\sum_{z\in\Omega}\left( \mu_{zx}+\sum_{y\in\delta\Omega}\frac{\mu_{yx}}{m'_y}\mu_{yz}\right)(u(z)-u(x))\\
&=&\frac{1}{\widetilde{m}_x}\sum_{z\in\Omega}\widetilde\mu_{zx}(u(z)-u(x)),
\end{eqnarray*} where we have used the definition of $\widetilde{\mu}$ and Proposition~\ref{prop:degree}. The lemma follows from calculation.
\end{proof}

\section{Neumann Cheeger constants}
In this section, we prove some useful properties of the Neumann Cheeger constant $h_N(\Omega),$ i.e. Theorem~\ref{cheegerseqsobolev} and Theorem~\ref{thm:main2}.

To simplify the notation, for any function $f\in\R^{\overline{\Omega}}$ and any constant $a\in\R,$ we denote by $\{f>a\}:=\{x\in\overline{\Omega}:f(x)>a\}$ the super-level set of $f$ in $\overline{\Omega}$, and by $\{f\geq a\},\{f<a\}$ and $\{f\leq a\}$ in the same way.

\begin{proof}[Proof of Theorem \ref{cheegerseqsobolev}]
For any nonconstant function $f\in\R^{\overline{\Omega}},$ we choose a constant $c\in\R$ such that
$$m(\{f<c\}\cap\Omega)\leq m(\{f\geq c\}\cap\Omega)$$
and
$$m(\{f\leq c\}\cap\Omega)\geq m(\{f> c\}\cap\Omega).$$
Set $g:=f-c $, then we have for any $\sigma \leq 0$,
$$m(\{g<\sigma\}\cap\Omega)\leq m(\{g\geq \sigma\}\cap\Omega)$$
and for any $\sigma >0$,
$$m(\{g<\sigma\}\cap\Omega)\geq m(\{g\geq \sigma\}\cap\Omega).$$
For any $\sigma\in\R$, we define
$$ G(\sigma):=\sum_{\substack{e=\{x,y\}\in E_{\Omega}\\ g(x)< \sigma \leq g(y)}}\mu_{xy}.$$

We claim that
$$\sum_{e=\{x,y\}\in E_{\Omega}}|f(x)-f(y)|\mu_{xy}=\int_{-\infty}^{\infty}G(\sigma)d\sigma,$$ which is a discrete version of the Coarea formula. For discrete Coarea formula, see e.g. \cite[Lemma 3.3]{Grigor'yan2011}.
By calculation,
\begin{eqnarray}
  \int_{-\infty}^{\infty}G(\sigma)d\sigma&=&\int_{-\infty}^{\infty}\sum_{\substack{e=\{x,y\}\in E_{\Omega},\\g(x)< \sigma \leq g(y)}}\mu_{xy}d\sigma=\int_{-\infty}^{\infty}\sum_{e=\{x,y\}\in E_{\Omega}}\mu_{xy}\cdot \chi_{(g(x),g(y)]}(\sigma)d\sigma\nonumber\\
  &=&\sum_{e=\{x,y\}\in E_{\Omega}}\mu_{xy}\int_{-\infty}^{\infty}\chi_{(g(x),g(y)]}(\sigma)d\sigma=\sum_{e=\{x,y\}\in  E_{\Omega}}|g(x)-g(y)|\mu_{xy}\nonumber\\
  &=&\sum_{e=\{x,y\}\in  E_{\Omega}}|f(x)-f(y)|\mu_{xy},\nonumber
\end{eqnarray}
where $\chi_{(g(x),g(y)]}(\cdot)$ is the characteristic function of the interval $(g(x),g(y)]$ in $\R.$ This proves the claim.
\par
We set $S:=\{g<\sigma\}$ for $\sigma\leq 0$ and $S:=\{g\geq \sigma\}$ for $\sigma> 0.$ In either case, $m(S\cap\Omega)\leq m(S^{\vee}\cap\Omega)$ and hence by the definition of $h_N(\Omega),$
$$G(\sigma)\geq h_N(\Omega)\cdot m(S\cap\Omega)=h_N(\Omega)\cdot\left\{
\begin{array}{ll}
m(\{g<\sigma\}\cap \Omega),& \mathrm{for}\ \sigma\leq 0, \\
m(\{g\geq\sigma\}\cap \Omega),& \mathrm{for}\ \sigma>0. \\
\end{array}\right.$$
\newline
So the previous claim yields
\begin{eqnarray}
&&\sum_{e=\{x,y\}\in E_{\Omega}}|f(x)-f(y)|\mu_{xy}=\int_{-\infty}^{0}G(\sigma)d\sigma+\int_{0}^{\infty}G(\sigma)d\sigma\nonumber\\
&\geq&h_N(\Omega)\left(\int_{-\infty}^{0}m(\{g<\sigma\}\cap \Omega)d\sigma+\int_{0}^{\infty}m(\{g\geq\sigma\}\cap \Omega)d\sigma\right)\nonumber\\
&=&h_N(\Omega)\left(\int_{-\infty}^{0}\sum_{x\in\Omega}\chi_{(g(x),0]}(\sigma)m_xd\sigma+
\int_{0}^{\infty}\sum_{x\in\Omega}\chi_{(0,g(x)]}(\sigma)m_xd\sigma\right)\nonumber\\
&=&h_N(\Omega)\sum_{x\in\Omega}|f(x)-c|m_x\geq h_N(\Omega)\inf_{c\in \R}\sum_{x\in\Omega}|f(x)-c|m_x,\nonumber
\end{eqnarray} where we have interchanged the summation and the integration in the last line. Taking the infimum over all nonconstant functions $f,$ we prove that
$$\inf_f\frac{\sum_{e=\{x,y\}\in E_{\Omega}}|f(x)-f(y)|\mu_{xy}}{\inf_{c\in\R} \sum_{x\in \Omega}|f(x)-c|m_{x}}\geq h_N(\Omega).$$

For the opposite direction, let $S\subset\overline\Omega$ attain the infimum in the definition of $h_{N}(\Omega),$ i.e.
$$m(S\cap\Omega)\leq m(S^{\vee} \cap \Omega),\quad h_N(\Omega)=\frac{\mu(\partial_\Omega S)}{m(S\cap\Omega)}.$$
Consider a function $\varphi$ on $\overline{\Omega}$ given by
\[
\varphi(x)=\left\{
       \begin{array}{ll}
        1,& x\in S, \\
        -1,&x\in S^{\vee}.
               \end{array}
     \right.
\]
Then \begin{eqnarray}
&&\frac{\sum_{e=\{x,y\}\in E_{\Omega}}|\varphi(x)-\varphi(y)|\mu_{xy}}{\inf_{c\in\R}\sum_{x\in\Omega}|\varphi(x)-c|m_x}\nonumber\\
&=&\sup_{c\in\R}\frac{2\mu(\partial_{\Omega}S)}{|1-c|m(S\cap\Omega)+|1+c|m(S^\vee\cap\Omega)}\nonumber\\
&=&\frac{2\mu(\partial_{\Omega}S)}{2m(S\cap\Omega)}=h_N(\Omega).\nonumber
\end{eqnarray}
Since $m(S\cap\Omega)\leq m(S^\vee \cap \Omega)$, the supremum for $c$ in the above attains at $c=-1$. This proves the other direction and hence the theorem.
\end{proof}

Recall that $\widetilde{\Omega}=(\Omega,\widetilde{\mu})$ is the modified graph of $\Omega$.
We denote by $\widetilde{E}$ the set of edges in $\widetilde{\Omega}$ induced by $\widetilde{\mu},$ i.e. $\{x,y\}\in \widetilde{E}$ if and only if $\widetilde{\mu}_{xy}>0,$ as usual. Given $K\subset\Omega,$ we denote by $\partial_{\widetilde\Omega}K$ the boundary of $K$ in $\widetilde\Omega$ and by $\widetilde{\mu}(\partial_{\widetilde\Omega}K)$ its measure. For convenience, for any nonempty proper subset $S$ of $\overline{\Omega},$ we simply write $\partial_{\widetilde\Omega}S:=\partial_{\widetilde\Omega}(S\cap \Omega)$ if it doesn't make any confusion. In addition, by Proposition \ref{prop:degree}, $\widetilde{m}$ coincides with $m$ on $\Omega.$

Now we are ready to prove Theorem~\ref{thm:main2}.
\begin{proof}[Proof of Theorem~\ref{thm:main2}]
For the lower bound estimate in the theorem, it suffices to show that for any nonempty proper subset $S$ in $\OM,$ $$\frac{\widetilde{\mu}(\partial_{\widetilde\Omega}S)}{\min\{m(S\cap\Omega),m( S^\vee\cap\Omega)\}}\leq\frac{\mu(\partial_{\Omega}S)}{\min\{m(S\cap\Omega),m( S^\vee\cap\Omega)\}},$$
i.e.
$\widetilde{\mu}(\partial_{\widetilde\Omega}S)\leq \mu(\partial_{\Omega}S).$
By the calculation,
\begin{eqnarray*}
&&\widetilde{\mu}(\partial_{\widetilde\Omega}S)=\sum_{\substack{x\in S\cap\Omega\\y\in S^\vee\cap\Omega}}\left(\mu_{xy}+\sum_{z\in\delta\Omega}\frac{\mu_{xz}\mu_{zy}}{m'_z}\right)\\
&=&\sum_{\substack{x\in S\cap\Omega\\y\in S^\vee\cap\Omega}}\mu_{xy}+\sum_{\substack{x\in S\cap\Omega\\y\in S^\vee\cap\Omega}}\left(\sum_{z\in S\cap\delta\Omega}+\sum_{z\in S^\vee\cap\delta\Omega}\right)\frac{\mu_{xz}\mu_{zy}}{m'_z}\\
&=&\sum_{\substack{x\in S\cap\Omega\\y\in S^\vee\cap\Omega}}\mu_{xy}+\sum_{\substack{z\in S\cap\delta\Omega\\y\in S^\vee\cap\Omega}}\left(\sum_{x\in S\cap\Omega}\frac{\mu_{xz}}{m'_z}\right)\mu_{zy}+\sum_{\substack{x\in S\cap\Omega\\z\in S^\vee\cap\delta\Omega}}\left(\sum_{y\in S^\vee\cap\Omega}\frac{\mu_{zy}}{m'_z}\right)\mu_{xz}\\
&\leq&\sum_{\substack{x\in S\cap\Omega\\y\in S^\vee\cap\Omega}}\mu_{xy}+\sum_{\substack{z\in S\cap\delta\Omega\\y\in S^\vee\cap\Omega}}\mu_{zy}+\sum_{\substack{x\in S\cap\Omega\\z\in S^\vee\cap\delta\Omega}}\mu_{xz}=\sum_{x\in S}\sum_{y\in S^\vee}\mu_{xy}=\mu(\partial_{\Omega}S).
\end{eqnarray*}
From this we prove the lower bound.

To show $h_{N}(\Omega)\leq 2h_{\widetilde\Omega},$ we assume $A$ is the subset of $\Omega$ that achieves $h_{\widetilde\Omega}$ and $B=\Omega\setminus A$. For simplicity, if $z\in\delta\Omega$, we set $m_A(z)=\sum_{x\in A}\mu_{xz}$ and $m_B(z)$ similarly. Let $K=\{z\in\delta\Omega|m_A(z)\neq0, m_B(z)\neq0\}$, set $S=A\cup\{z\in\delta\Omega|m_B(z)=0\}\cup\{z\in K|m_A(z)> m_B(z)\}\cup C$, where $C$ is any subset of $\{z\in K:m_A(z)=m_B(z)\}$, then $S^\vee=\overline{\Omega}\setminus S$. Hence $\overline\Omega$ is divided into two parts $S$ and $S^\vee$.

By calculation,
\begin{eqnarray*}
&&\mu(\partial_{\Omega}S)=\sum_{\{x,y\}\in E(A,B)}\mu_{xy}+\sum_{z\in K\cap S^\vee}m_A(z)+\sum_{z\in K\cap S}m_B(z)\\
&\leq&2\sum_{\{x,y\}\in E(A,B)}\mu_{xy}+2\sum_{z\in K\cap S^\vee}\frac{m_A(z)\cdot m_B(z)}{m'_z}+2\sum_{z\in K\cap S}\frac{m_A(z)\cdot m_B(z)}{m'_z}\\
&=&2\sum_{x\in A}\sum_{y\in B}\mu_{xy}+2\sum_{x\in A}\sum_{y\in B}\sum_{z\in K\cap S^\vee}\frac{\mu_{xz}\mu_{zy}}{m'_z}+2\sum_{x\in A}\sum_{y\in B}\sum_{z\in K\cap S}\frac{\mu_{xz}\mu_{zy}}{m'_z}\\
&=&2\sum_{x\in A}\sum_{y\in B}\left(\mu_{xy}+\sum_{z\in K}\frac{\mu_{xz}\mu_{zy}}{m'_z}\right)=2\sum_{x\in A}\sum_{y\in B}\widetilde\mu_{xy}
\end{eqnarray*}
Hence
$$h_{N}(\Omega)\leq\frac{\mu(\partial_{\Omega}S)}{\min\{m(S\cap\Omega),m( S^\vee\cap\Omega)\}}\leq 2\frac{\sum_{x\in A}\sum_{y\in B}\widetilde\mu_{xy}}{\min\{m(A),m(B)\}}=2h_{\widetilde\Omega}.$$
\end{proof}

Now, we discuss the conditions for the equalities in Theorem~\ref{thm:main2}. First, we introduce some notations. We denote a partition of $\Omega$ (resp. $\overline{\Omega}$) by $P_A=\{A,B\}$ (resp. $\overline{P}_S=\{S,S^{\vee}\}$), where $A\sqcup B=\Omega$ (resp. $S\sqcup S^{\vee}=\overline{\Omega}$), here $\sqcup$ denotes disjoint union. For convenience, we introduce notations
\begin{equation}\label{inducedpartition}
\overline{P}_S\cap\Omega:=\{S\cap\Omega,S^{\vee}\cap\Omega\}=P_{S\cap\Omega},
\end{equation}
which is a partition of $\Omega$ induced by $\overline{P}_S$.
The set of partitions of $\Omega$ (resp. $\overline\Omega$) that achieves $h_{\widetilde{\Omega}}$ (resp. $h_N(\Omega)$) is denoted by $\mathscr{P}$ (resp. $\overline{\mathscr{P}}$).

\begin{figure}[!h]
\includegraphics[height=5cm,width=14cm]{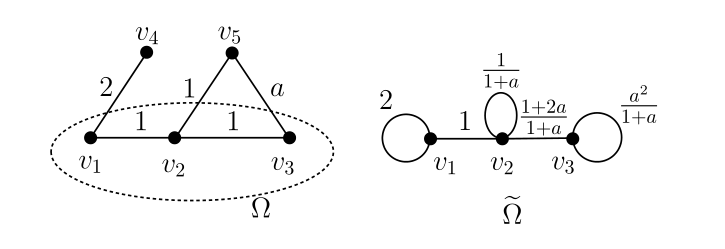}
\caption{}\label{figure6}
\end{figure}
A natural question is that what's the relation between $\mathscr{P}$ and $\overline{\mathscr{P}}\cap\Omega$, where $\overline{\mathscr{P}}\cap\Omega:=\{\overline{P}_S\cap\Omega\mid\overline{P}_S\in\overline{\mathscr{P}}\}$ and $\overline{P}_S\cap\Omega$ is defined in \eqref{inducedpartition}. The following example shows that $\mathscr{P}$ and $\overline{\mathscr{P}}\cap\Omega$ are different in general.

\begin{example}
If we set $a=\frac{9}{2}$ (the weight of edge $\{v_3,v_5\}$) in the graph as shown in Figure \ref{figure6}. By calculation, we have $h_N(\Omega)=\frac{1}{3}$ and $\overline{\mathscr{P}}=\{\{\{v_1,v_4\},\{v_2,v_3,v_5\}\}\}$, $h_{\widetilde{\Omega}}=\frac{40}{121}$ and $\mathscr{P}=\{\{\{v_1,v_2\},\{v_3\}\}\}$. Hence $\overline{\mathscr{P}}\cap\Omega\nsubseteq\mathscr{P}$ and $\mathscr{P}\nsubseteq\overline{\mathscr{P}}\cap\Omega$.
\end{example}

Here we rewrite the definitions of $h_N(\Omega)$ and $h_{\widetilde{\Omega}}$ by partitions of $\overline{\Omega}$ and $\Omega$ respectively. Set $$\eta(\overline{P}_S):=\frac{\mu(\partial_{\Omega}S)}{\min\{m(S\cap\Omega),m(S^{\vee}\cap\Omega)\}},$$
$$\zeta(P_A):=\frac{\widetilde{\mu}(\partial_{\widetilde{\Omega}}A)}{\min\{m(A),m(B)\}},$$
then
$$h_N(\Omega)=\inf_{\overline{P}_S}\eta(\overline{P}_S),$$
$$h_{\widetilde{\Omega}}=\inf_{P_A}\zeta(P_A).$$
To characterise the equalities in Theorem \ref{thm:main2}, we divide it into two cases.

Case 1: $h_{\widetilde{\Omega}}\leq h_{N}(\Omega).$

For any $\emptyset\neq S\subset\overline{\Omega},$ set $A=S\cap\Omega$ and $B=S^{\vee}\cap\Omega.$ From the proof of Theorem~\ref{thm:main2}, we have $\widetilde{\mu}(\partial_{\widetilde{\Omega}}S)=\mu(\partial_{\Omega}S)$ $\Longleftrightarrow$ $m_A(z)=0$, $\forall z\in S^{\vee}\cap\delta\Omega$ and  $ m_B(z)=0$, $\forall z\in S\cap\delta\Omega.$
\begin{prop}\label{firstequality}
$h_{\widetilde{\Omega}}= h_{N}(\Omega)$ if and only if the following two properties hold

(1) $\mathscr{P}\cap(\overline{\mathscr{P}}\cap\Omega)\neq\emptyset$;

(2) For any $ P_{A}=\{A,B\}\in \mathscr{P}\cap(\overline{\mathscr{P}}\cap\Omega)$, $m_A(z)=0$ or $m_B(z)=0,$ $\forall z\in\delta\Omega.$
\end{prop}

\begin{proof}
The if part is easy, we only need to show the only if part.

Choose $\overline{P}_S\in\overline{\mathscr{P}}$, then $P_{S\cap\Omega}:=\overline{P}_S\cap\Omega$ is a partition of $\Omega$ and we have
$$h_{\widetilde{\Omega}}=h_N(\Omega)=\eta(\overline{P}_S)\geq \zeta(P_{S\cap\Omega}).$$
The above inequality is obtained by the same process as in the proof of the lower bound estimate in Theorem \ref{thm:main2}. By the minimum of $h_{\widetilde{\Omega}}$, we have $h_N(\Omega)=\zeta(P_{S\cap\Omega})$, i.e. $\widetilde{\mu}(\partial_{\widetilde{\Omega}}S)=\mu(\partial_{\Omega}S)$. Hence $m_A(z)=0$, $\forall z\in S^{\vee}\cap\delta\Omega$ and  $ m_B(z)=0$, $\forall z\in S\cap\delta\Omega.$
\end{proof}
From the proof of Proposition \ref{firstequality}, we easily have
\begin{coro}
If $h_{\widetilde{\Omega}}=h_N(\Omega)$, then $\overline{\mathscr{P}}\cap\Omega\subseteq\mathscr{P}$.
\end{coro}
From the following example, one can see that even for the equality case, the property $\mathscr{P}\subseteq\overline{\mathscr{P}}\cap\Omega$ does not always hold.
\begin{example}
If we set $a=2+\sqrt{6}$ in the graph as shown in Figure \ref{figure6}, then we have $h_{\widetilde{\Omega}}=h_N(\Omega)=\frac{1}{3}$. Moreover, we have $\overline{\mathscr{P}}=\{\{\{v_1,v_4\},\{v_2,v_3,v_5\}\}\}$ and $\mathscr{P}=\{\{\{v_1,v_2\},\{v_3\}\},\{\{v_1\},\{v_2,v_3\}\}\}$, one can see that $\mathscr{P}\nsubseteq\overline{\mathscr{P}}\cap\Omega$.
\end{example}

Case 2: $h_{N}(\Omega)\leq 2h_{\widetilde{\Omega}}.$

Choose $P_{A}\in\mathscr{P}$. From the proof of Theorem~\ref{thm:main2}, we have a partition of $\overline\Omega$ denoted by $\overline{P}_{S}=\{S,S^{\vee}\}$ induced from $P_{A}$, where $S\cap\Omega=A$ and $\eta(\overline{P}_S)=2h_{\widetilde{\Omega}}$ if and only if

(1) $E(A,B)=\emptyset$ and (2) For any $z\in\delta\Omega$ one of the following holds:

\quad (a) $m_A(z)=0$,

\quad (b) $m_B(z)=0$,

\quad (c) $m_A(z)=m_B(z)$.
\begin{prop}\label{secondequality}
$h_{N}(\Omega)= 2h_{\widetilde{\Omega}}$ if and only if the following two properties hold

(1) $\mathscr{P}\cap(\overline{\mathscr{P}}\cap\Omega)\neq\emptyset$;

(2) For any $ P_{A}\in\mathscr{P}\cap(\overline{\mathscr{P}}\cap\Omega)$, $E(A,B)=\emptyset$ and for any $z\in\delta\Omega$ one of $m_A(z)=0,$ $m_B(z)=0,$ $m_A(z)=m_B(z)$ holds.
\end{prop}
\begin{proof}
The if part is easy, we only need to show the only if part.

Choose $P_{A}\in\mathscr{P}$. Then we have
$$h_{N}(\Omega)= 2h_{\widetilde{\Omega}}\geq \eta(\overline{P}_S),$$
where $\overline{P}_S$ is a partition of $\overline{\Omega}$ induced by $P_A$ and the above inequality is obtained by the same process as in the proof of the upper bound estimate in Theorem \ref{thm:main2}.

By the minimum of $h_{N}$, we have $2h_{\widetilde{\Omega}}=\eta(\overline{P}_S)$. Hence the proposition follows.
\end{proof}
From the proof of Proposition \ref{secondequality}, one can easily get
\begin{coro}
If $h_{N}(\Omega)= 2h_{\widetilde{\Omega}}$, then $\mathscr{P}\subseteq (\overline{\mathscr{P}}\cap\Omega)$.
\end{coro}
\begin{figure}[!h]
\includegraphics[height=5cm,width=13cm]{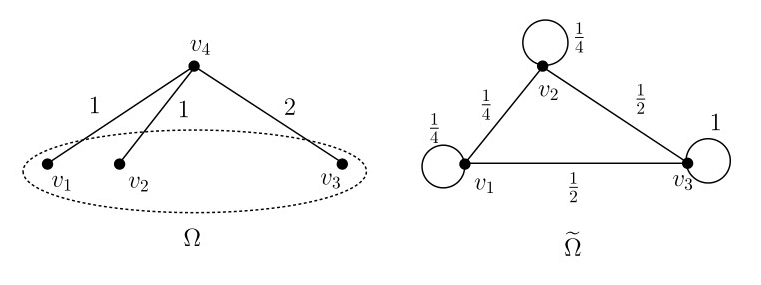}
\caption{}\label{figure4}
\end{figure}
We can see that even for the equality case, the property $\overline{\mathscr{P}}\cap\Omega\subseteq\mathscr{P}$ doesn't always hold, see the following example.

\begin{example}
Let $\Omega=\{v_1,v_2,v_3\}$. The graph and edge weights are as shown in Figure \ref{figure4}.
Then we have $h_{\widetilde{\Omega}}=\frac12$ and $h_N(\Omega)=1.$ If we choose $\overline{\mathscr{P}}\ni \overline{P}_{S}$, where $S=\{v_1\}$. Then $P_{S\cap\Omega}$ is a partition of $\Omega$, but $\zeta(P_{S\cap\Omega})=\frac34$.
\end{example}

\section{neumann cheeger estimate}
In this section, we prove the main result, i.e. Theorem \ref{NeumannCheegerestimates}, following the arguments in \cite[Lemma 2.1, Theorem 2.3]{Chung97}.
\begin{proof}[Proof of Theorem \ref{NeumannCheegerestimates}]
For the upper bound estimate $\lambda_{1,N}(\Omega)\leq 2h_N(\Omega).$
We choose $S\subset\overline{\Omega}$ achieves $h_N(\Omega).$ Set $f$ as
\[
f(x)=\left\{
       \begin{array}{ll}
        \frac{1}{m(S\cap\Omega)},& x\in S, \\
        -\frac{1}{m(S^{\vee}\cap\Omega)},&x\in S^{\vee}.
               \end{array}
     \right.
\]
By substituting $f(x)$ into \eqref{Reighley quotient}, we have
\begin{eqnarray*}
\lambda_{1,N}(\Omega)&\leq& (\frac{1}{m(S\cap\Omega)}+\frac{1}{m(S^{\vee}\cap\Omega)})\mu(\partial_{\Omega} S)\\
&\leq&\frac{2\mu(\partial_{\Omega} S)}{\min\{m(S\cap\Omega),m(S^{\vee}\cap\Omega)\}}\\
&=&2h_N(\Omega).
\end{eqnarray*}

Now we turn to the lower bound estimate, i.e. $2-\sqrt{4-h_N^2(\Omega)}\leq\lambda_{1,N}(\Omega).$

Let $0\not\equiv u\in\R^{\OM}$ be the first eigenfunction associated with the first eigenvalue $\lambda_{1,N}(\Omega),$ i.e. $u$ satisfied:
\begin{equation}\label{eq:eigenequation}
 -\Delta u(x)=\lambda_{1,N}(\Omega) u(x), x\in \Omega.
\end{equation}

By $\langle u,1\rangle_{\Omega}=0,$ $\{x\in\Omega| u(x)>0\}$ and $\{x\in\Omega| u<0\}$ are nonempty. Without loss of generality, we may assume that \begin{equation}\label{eq:eq1}m(\{x\in\Omega| u(x)>0\})\leq m(\{x\in\Omega| u<0\}),\end{equation} otherwise consider $-u$.

  For the eigenfunction $u,$ we set $\OM_+=\{x\in\OM| u(x)\geq 0\},$ $\OM_-=\{x\in\OM| u(x)< 0\},$ $\Omega_+=\OM_+\cap\Omega$ and $\Omega_-=\OM_-\cap\Omega.$

  Multiplying the both sides of the equation in \eqref{eq:eigenequation} by $u(x),$ for $x\in\Omega_+,$ and summing over $x\in\Omega_+$ w.r.t. the measure $m,$ we have \begin{equation*}\sum_{x\in\Omega_+}\sum_{y\in\OM}\mu_{xy}(u(x)-u(y))u(x)=\lambda_{1,N}(\Omega)\sum_{x\in\Omega_+}u^2(x)m_x. \end{equation*} By the Neumann boundary condition of $u$ and note that here we consider graph $(\overline\Omega,\mu)$ with edges in $E_{\Omega},$
  $$\sum_{x\in\OM_+\setminus\Omega_+}\sum_{y\in\OM}\mu_{xy}(u(x)-u(y))u(x)=0.$$
  Adding up the last two equations, we have
 $$\sum_{x\in\OM_+}\sum_{y\in\OM}\mu_{xy}(u(x)-u(y))u(x)=\lambda_{1,N}(\Omega)\sum_{x\in\Omega_+}u^2(x)m_x.$$
We set
 \[
v(x)=u^+(x)=\left\{
       \begin{array}{ll}
        u(x),& x\in \OM_+, \\
        0,&x\in \OM_-.
               \end{array}
     \right.
\]
Then
\begin{eqnarray*}
&&\sum_{x\in\OM_+}\sum_{y\in\OM}\mu_{xy}(u(x)-u(y))u(x)\\
&=&\sum_{x,y\in\OM_+}\mu_{xy}(u(x)-u(y))u(x)+\sum_{x\in\OM_+}\sum_{y\in\OM_-}\mu_{xy}(u(x)-u(y))u(x)\\
&\geq&\sum_{x,y\in\OM_+}\mu_{xy}(v(x)-v(y))v(x)+\sum_{x\in\OM_+}\sum_{y\in\OM_-}\mu_{xy}(v(x)-v(y))^2\\
&=&\frac{1}{2}\sum_{x,y\in\OM_+}\mu_{xy}(v(x)-v(y))^2+\sum_{x\in\OM_+}\sum_{y\in\OM_-}\mu_{xy}(v(x)-v(y))^2
\end{eqnarray*}

Hence
\begin{eqnarray}\label{eq:eq3}
\lambda_{1,N}=\frac{D_\Omega(v)}{\sum_{x\in\Omega}v^2(x)m_x}=:W
\end{eqnarray}

To estimate $W,$
\begin{eqnarray*}
  W&=&\frac{\sum_{e=\{x,y\}\in E_{\Omega}}\mu_{xy}(v(x)-v(y))^2
  \sum_{e=\{x,y\}\in E_{\Omega}}\mu_{xy}(v(x)+v(y))^2}{\sum_{x\in\Omega}v^2(x)m_x
  \sum_{e=\{x,y\}\in E_{\Omega}}\mu_{xy}(v(x)+v(y))^2}=\frac{I}{II}.\\
\end{eqnarray*} For the term $I,$ the H\"older's inequality yields
\begin{eqnarray}\label{eq:eq6}
  I^{\frac{1}{2}}\geq \sum_{\substack{e=\{x,y\}\in E_{\Omega}\\v(x)>v(y)}}\mu_{xy}(v(x)^2-v(y)^2)
\end{eqnarray} It suffices to estimate the term in the above bracket. Let $\chi_{[a,b)}(t)$ be the characteristic function on interval $[a,b)$.
For any $t\geq0,$ set $P_t:=\{x\in\OM| v(x)>t\}.$
Using Fubini's theorem to change the order of integrals and summations, we have
\begin{eqnarray}
  &&\sum_{\substack{e=\{x,y\}\in E_{\Omega}\\v(x)>v(y)}}\mu_{xy}(v(x)^2-v(y)^2)=2\sum_{\substack{e=\{x,y\}
  \in E_{\Omega}\\v(x)>v(y)}}\mu_{xy}\int_{v(y)}^{v(x)}tdt\nonumber\\
&=&2\sum_{\substack{e=\{x,y\}\in E_{\Omega}\\v(x)>v(y)}}\mu_{xy}\int_{0}^{\infty}t\cdot\chi_{[v(y),v(x))}(t) dt\nonumber\\
&=&2\int_0^\infty tdt\sum_{\substack{e=\{x,y\}\in E_{\Omega}\\v(x)>t\geq v(y)}}\mu_{xy}=2\int_0^\infty tdt\mu(\partial_{\Omega}P_t)\nonumber\\
&\geq&2h_N(\Omega)\int_0^\infty tdtm(P_t\cap\Omega)=2h_N(\Omega)\int_0^\infty tdt\sum_{\substack{x\in\Omega\\v(x)>t}}m_x\label{eq:eq2}\\
&=&2h_N(\Omega)\sum_{x\in\Omega}m_x\int_0^{v(x)}tdt=h_N(\Omega)\sum_{x\in\Omega}v^2(x)m_x,\label{eq:eq5}\end{eqnarray}
where we have used the Cheeger constant $h_N(\Omega)$ in \eqref{eq:eq2} since $m(P_t\cap\Omega)\leq m(P_t^{\vee}\cap\Omega)$ which follows from our assumption \eqref{eq:eq1}.

To estimate the term $II,$ by \eqref{eq:eq3} we have
\begin{eqnarray}\label{eq:eq4}
  &&\sum_{e=\{x,y\}\in E_{\Omega}}\mu_{xy}(v(x)+v(y))^2\nonumber\\&=&2\sum_{e=\{x,y\}\in E_{\Omega}}\mu_{xy}(v(x)^2+v(y)^2)-
  \sum_{e=\{x,y\}\in E_{\Omega}}\mu_{xy}(v(x)-v(y))^2\nonumber\\
  &=&\sum_{x,y\in\Omega}\mu_{xy}(v(x)^2+v(y)^2)+2\sum_{x\in\Omega,z\in\delta\Omega}\mu_{xz}(v(x)^2+v(z)^2)
  -W\cdot\sum_{x\in\Omega_+}v^2(x)m_x
\nonumber\\
&=&2\sum_{x\in\Omega,y\in\OM}\mu_{xy}v(x)^2+2\sum_{x\in\Omega,z\in\delta\Omega}\mu_{xz}v(z)^2-W\cdot\sum_{x\in\Omega}v^2(x)m_x\end{eqnarray}
Note that the Neumann condition for $u$ yields that for any $z\in\delta\Omega,$
$u(z)=\frac{1}{m'_z}\sum_{w\in\Omega}\mu_{zw}u(w).$ This implies that
$$v(z)\leq \frac{1}{m'_z}\sum_{w\in\Omega}\mu_{zw}v(w).$$ Hence by H\"older's inequality,
\begin{eqnarray*}\sum_{x\in\Omega,z\in\delta\Omega}\mu_{xz}v(z)^2&\leq&\sum_{x\in\Omega,z\in\delta\Omega}\frac{\mu_{xz}}{m'_z}\sum_{w\in\Omega}\mu_{zw}v(w)^2\\
&\leq&\sum_{w\in\Omega,z\in\delta\Omega}\mu_{zw}v(w)^2\\
&\leq&\sum_{w\in\Omega}v(w)^2m_w.
\end{eqnarray*} Noting that \eqref{eq:eq4}, we have
\begin{equation*}\sum_{e=\{x,y\}\in E_{\Omega}}\mu_{xy}(v(x)+v(y))^2\leq 4\sum_{x\in\Omega}v(x)^2m_x-W\cdot\sum_{x\in\Omega}v^2(x)m_x.\end{equation*}
Combining this with the estimates in \eqref{eq:eq6} and \eqref{eq:eq5}, we get
\begin{eqnarray*}
  W= \frac{I}{II}\geq h_N(\Omega)^2\frac{1}{4-W}.
\end{eqnarray*}
This yields that $$\lambda_{1,N}(\Omega)\geq W\geq 2-\sqrt{4-h_N(\Omega)^2}.$$
\end{proof}

At last, we give another example to compare the lower estimate in \eqref{eq:main estimate} and \eqref{eq:trivial estimate}.
\begin{figure}[!h]
\includegraphics[height=5cm,width=10cm]{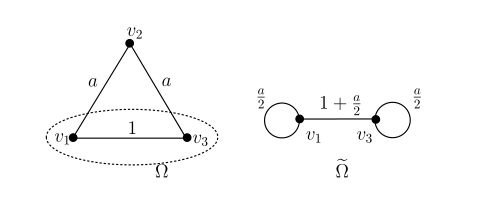}
\caption{}\label{figure5}
\end{figure}
\begin{example}
For any $a>0,$ let $G_a=(\{v_1,v_2,v_3\},\mu_a)$ be a triangle with edge weights $\mu_{a}(\{v_1,v_2\})=\mu_{a}(\{v_2,v_3\})=a$ and $\mu_{a}(\{v_1,v_3\})=1$ as shown in Figure \ref{figure5}. Let $\Omega=\{v_1,v_3\}.$ Suppose $a\gg 1$, then $h_{N}(\Omega)=1$ and $h_{\widetilde\Omega}=\frac{1+\frac{a}{2}}{1+a}$. Hence $2-\sqrt{4-h_N(\Omega)^2}=2-\sqrt{3}$ and $1-\sqrt{1-h_{\widetilde\Omega}^2}\rightarrow\frac{2-\sqrt{3}}{2}, $ as $a\rightarrow+\infty$. By calculation, our lower bound estimate in \eqref{eq:main estimate} is better than \eqref{eq:trivial estimate} when $a\geq1.76$.
\end{example}

\textbf{Acknowledgements.} B.H. is supported by NSFC, grant no. 11401106.

\bibliography{NeumannCondition}
\bibliographystyle{alpha}

\end{document}